\documentclass[12pt]{amsart}

\usepackage{amsmath, amssymb}
\usepackage[all]{xy}

\theoremstyle{plain}
\newtheorem{prop}{Proposition}[section]

\let\a\alpha
\let\b\beta

\let\d\delta
\let\e\varepsilon

\let\f\varphi

\let\Om\Omega

\def\lra{\longrightarrow}
\def\egal{\ar@{=}}

\def\A{\mathcal A}

\def\C{\mathbb C}
\def\CC{\mathcal C}
\def\E{\mathcal E}
\def\F{\mathcal F}
\def\G{\mathcal G}

\def\I{\mathcal I}
\def\J{\mathcal J}
\def\K{\mathcal K}
\def\P{\mathbb P}

\def\O{\mathcal O}
\def\TT{\mathcal T}

\def\W{\mathbb W}

\def\Ker{{\mathcal Ker}}
\def\Coker{{\mathcal Coker}}
\def\Im{{\mathcal Im}}

\def\D{{\scriptscriptstyle \operatorname{D}}}
\def\T{{\scriptscriptstyle \operatorname{T}}}
\def\EE{\operatorname{E}}
\def\H{\operatorname{H}}
\def\h{\operatorname{h}}
\def\M{\operatorname{M}_{{\mathbb P}^2}}
\def\N{\operatorname{N}}
\def\pp{\operatorname{p}}
\def\PP{\operatorname{P}}
\def\SS{\operatorname{S}}
\def\Hom{\operatorname{Hom}}
\def\Aut{\operatorname{Aut}}
\def\Ext{\operatorname{Ext}}

\def\rank{\operatorname{rank}}

\def\tensor{\otimes}
\def\isom{\simeq}

\newcommand{\noi}{\noindent}

\def\ds{\displaystyle}
\def\ba{\begin{array}}
\def\ea{\end{array}}

\begin{document}

\subjclass{Primary 14D20, 14D22}

\title[moduli of plane semi-stable sheaves with Hilbert polynomial $\operatorname{P}(m)=6m+2$]
{on the moduli space of semi-stable plane sheaves
with Hilbert polynomial $\mathbf{\operatorname{\mathbf{P}}(m)=6m+2}$}

\author{mario maican}

\address{Mario Maican \\
Institute of Mathematics of the Romanian Academy \\
Calea Grivi\c tei 21 \\
 010702 Bucharest \\
Romania}

\email{mario.maican@imar.ro}

\begin{abstract}
We study the Simpson moduli space of semi-stable sheaves on the complex
projective plane that have dimension $1$, multiplicity $6$ and Euler characteristic $2$.
We describe concretely these sheaves as cokernels of morphisms of locally free sheaves
and we stratify the moduli space according to the types of sheaves that occur.
\end{abstract}

\maketitle

\tableofcontents

\noi
{\sc Acknowledgements.} The author was supported by the
Consiliul Na\c tional al Cercet\u arii \c Stiin\c tifice,
an agency of the Romanian Government,
grant PN II--RU 169/2010 PD--219.


\section{Introduction}

\noi
Let $\M(r,\chi)$ denote the moduli space of Gieseker semi-stable sheaves on $\P^2(\C)$
with Hilbert polynomial $\PP(m)=rm+\chi$, $r$ and $\chi$ being fixed integers, $r \ge 1$.
Le Potier \cite{lepotier} found that $\M(r,\chi)$ is an irreducible projective variety of dimension $r^2+1$,
smooth at points given by stable sheaves and rational if $\chi \equiv 1$ or $2 \mod r$.
In \cite{drezet-maican}, \cite{mult_five} and \cite{mult_six_one} a complete description of semi-stable
sheaves giving points in $\M(4,\chi)$, $\M(5,\chi)$ and $\M(6,1)$ was found.
These moduli spaces were shown to have natural stratifications given by cohomological conditions
on the sheaves involved. Here we are concerned with $\M(6,2)$. We describe all semi-stable
sheaves giving points in $\M(6,2)$ and we decompose this moduli space into five strata:
an open stratum $X_0$; a locally closed stratum that is the disjoint union of two irreducible
locally closed subsets $X_1$ and $X_2$, each of codimension $3$; a locally closed stratum that is
the disjoint union of two irreducible locally closed subsets $X_3$ and $X_4$, each of codimension $5$;
an irreducible locally closed stratum $X_5$ of codimension $7$ and a closed irreducible stratum
$X_6$ of codimension $9$.
For some of these sets we have concrete geometric descriptions:
$X_1$ is a certain open subset inside a fibre bundle with fibre $\P^{20}$ and base $\N(3,4,3) \times \P^2$,
where $\N(3,4,3)$ is the moduli space of semi-stable Kronecker modules $f \colon 4\O(-2) \to 3\O(-1)$;
$X_3$ is an open subset of a fibre bundle with fibre $\P^{22}$ and base $Y \times \N(3,2,3)$,
where $Y$ is the Hilbert scheme of zero-dimensional subschemes of $\P^2$ of length $2$
and $\N(3,2,3)$ is the moduli space of semi-stable Kronecker modules $f \colon 2\O(-1) \to 3\O$;
$X_5$ is an open subset of a fibre bundle with fibre $\P^{24}$ and base $\P^2 \times Y$;
the closed stratum $X_6$
is isomorphic to the universal sextic in $\P^2 \times \P(\SS^6 V^*)$.
The following table contains a description of each $X_i$ by cohomological conditions.
The third column of the table lists all sheaves giving points in $X_i$.
The sets $W$ of morphisms $\f$ are acted upon by the algebraic groups of automorphisms
of sheaves and in each case, apart from $X_0$, the geometric quotient is $X_i$.
The points given by properly semi-stable sheaves are all in $X_0$, which is why
this stratum cannot be a geometric quotient of the set of morphisms.
The table below is organised as the table in the introduction to \cite{mult_six_one},
to which we generally refer for notations and conventions.

\begin{table}[!hpt]{}
\begin{center}
{\small
\begin{tabular}{|c|c|c|}
\hline \hline
{}
&
\begin{tabular}{c}
{\tiny cohomological} \\
{\tiny conditions}
\end{tabular}
&
$W$
\\
\hline
$X_0$
&
\begin{tabular}{r}
{} \\
$\h^0(\F(-1))=0$ \\
$\h^1(\F)=0$\\
$\h^0(\F \tensor \Om^1(1))=0$ \\
{}
\end{tabular}
&
\begin{tabular}{c}
{} \\
$0 \lra 4\O(-2) \stackrel{\f}{\lra} 2\O(-1) \oplus 2\O \lra \F \lra 0$ \\
$\f$ is not equivalent to a morphism of any of the forms \\
${\ds
\left[
\ba{cccc}
\star & 0 & 0 & 0 \\
\star & \star & \star & \star \\
\star & \star & \star & \star \\
\star & \star & \star & \star
\ea
\right], \quad \left[
\ba{cccc}
\star & \star & 0 & 0 \\
\star & \star & 0 & 0 \\
\star & \star & \star & \star \\
\star & \star & \star & \star
\ea
\right], \quad \left[
\ba{cccc}
\star & \star & \star & 0 \\
\star & \star & \star & 0 \\
\star & \star & \star & 0 \\
\star & \star & \star & \star
\ea
\right]
}$ \\
{}
\end{tabular} \\
\hline
$X_1$
&
\begin{tabular}{r}
{} \\
$\h^0(\F(-1))=0$ \\
$\h^1(\F)=0$\\
$\h^0(\F \tensor \Om^1(1))=1$ \\
{}
\end{tabular}
&
\begin{tabular}{c}
{} \\
$0 \lra 4\O(-2) \oplus \O(-1) \stackrel{\f}{\lra} 3\O(-1) \oplus 2\O \lra \F \lra 0$ \\
$\f_{12}=0$, $\f_{11}$ and $\f_{22}$ are semi-stable as Kronecker modules \\
{}
\end{tabular} \\
\hline
$X_2$
&
\begin{tabular}{r}
{} \\
$\h^0(\F(-1))=0$ \\
$\h^1(\F)=1$\\
$\h^0(\F \tensor \Om^1(1))=1$ \\
{}
\end{tabular}
&
\begin{tabular}{c}
{} \\
$0 \lra \O(-3) \oplus \O(-2) \oplus \O(-1) \stackrel{\f}{\lra} 3\O \lra \F \lra 0$ \\
$\f$ is not equivalent to a morphism of any of the forms \\
${\ds
\left[
\ba{ccc}
\star & \star & \star \\
\star & \star & 0 \\
\star & \star & 0
\ea
\right], \quad \left[
\ba{ccc}
\star & \star & \star \\
\star & 0 & \star \\
\star & 0 & \star
\ea
\right], \quad \left[
\ba{ccc}
\star & \star & \star \\
\star & \star & \star \\
\star & 0 & 0
\ea
\right]
}$ \\
{}
\end{tabular} \\
\hline
$X_{3}$
&
\begin{tabular}{r}
{} \\
$\h^0(\F(-1))=0$ \\
$\h^1(\F)=1$\\
$\h^0(\F \tensor \Om^1(1))=2$ \\
{}
\end{tabular}
&
\begin{tabular}{c}
{} \\
$0 \lra \O(-3) \oplus \O(-2) \oplus 2\O(-1) \stackrel{\f}{\lra} \O (-1) \oplus 3\O \lra \F \lra 0$ \\
$\f_{13}=0$, $\f_{12} \neq 0$ and does not divide $\f_{11}$ \\
$\f_{23}$ has linearly independent maximal minors \\
{}
\end{tabular} \\
\hline
$X_{4}$
&
\begin{tabular}{r}
{} \\
$\h^0(\F(-1))=1$ \\
$\h^1(\F)=1$\\
$\h^0(\F \tensor \Om^1(1))=3$ \\
{}
\end{tabular}
&
\begin{tabular}{c}
{} \\
$0 \lra \O(-3) \oplus 2\O(-2) \stackrel{\f}{\lra} 2\O(-1) \oplus \O(1) \lra \F \to 0$ \\
$\f$ is not equivalent to a morphism of any of the forms \\
${\ds
\left[
\ba{ccc}
\star & 0 & 0 \\
\star & \star & \star \\
\star & \star & \star
\ea
\right], \quad \left[
\ba{ccc}
\star & \star & 0 \\
\star & \star & 0 \\
\star & \star & \star
\ea
\right], \quad \left[
\ba{ccc}
0 & 0 & \star \\
\star & \star & \star \\
\star & \star & \star
\ea
\right], \quad \left[
\ba{ccc}
0 & \star & \star \\
0 & \star & \star \\
\star & \star & \star
\ea
\right]
}$ \\
{}
\end{tabular} \\
\hline
$X_{5}$
&
\begin{tabular}{r}
{} \\
$\h^0(\F(-1))=1$ \\
$\h^1(\F)=2$\\
$\h^0(\F \tensor \Om^1(1))=4$ \\
{}
\end{tabular}
&
\begin{tabular}{c}
{} \\
$0 \lra 2\O(-3) \oplus \O(-1) \stackrel{\f}{\lra} \O(-2) \oplus \O \oplus \O(1) \lra \F \lra 0$ \\
$\f_{11}$ has linearly independent entries \\
$\f_{22} \neq 0$ and does not divide $\f_{32}$ \\
{}
\end{tabular} \\
\hline
$X_6$
&
\begin{tabular}{r}
{} \\
$\h^0(\F(-1))=2$ \\
$\h^1(\F)=3$ \\
$\h^0(\F \tensor \Om^1(1))=6$ \\
{}
\end{tabular}
&
\begin{tabular}{c}
{} \\
$0 \lra \O(-4) \oplus \O \lra 2\O(1) \lra \F \lra 0$ \\
$\f_{12}$ has linearly independent entries \\
{}
\end{tabular} \\
\hline \hline
\end{tabular}
}
\end{center}
\end{table}


Let $C \subset \P^2$ denote an arbitrary smooth sextic curve and let $P_i$ denote distinct points on $C$.
The generic sheaves in $X_1$ are of the form $\O_C(1)(P_1+ \cdots +P_6-P_7)$,
where $P_1, \ldots, P_6$ are not contained in a conic curve.
The generic sheaves in $X_3$ have the form $\O_C(2)(-P_1-P_2-P_3+P_4+P_5)$,
where $P_1, P_2, P_3$ are non-colinear.
The generic sheaves in $X_4$ are of the form $\O_C(1)(P_1 + \cdots + P_5)$,
where $P_1, \ldots, P_5$ are in general linear position.
The generic sheaves in $X_5$ are of the form $\O_C(2)(P_1-P_2-P_3)$.
The sheaves giving points in $X_6$ are of the form $\O_C(2)(-P)$,
(in this case $C$ need not be smooth).


\section{The open stratum}

\begin{prop}
\label{2.1}
Every sheaf $\F$ giving a point in $\M(6,2)$ and satisfying the condition $\h^1(\F)=0$
also satisfies the condition $\h^0(\F(-1))=0$.
For these sheaves $\h^0(\F \tensor \Om^1(1))= 0$ or $1$.
The sheaves from the first case are given by resolutions of the form
\[
\tag{i}
0 \lra 4\O(-2) \stackrel{\f}{\lra} 2\O(-1) \oplus 2\O \lra \F \lra 0,
\]
where $\f$ is not equivalent, modulo the action of the natural group of automorphisms,
to a morphism represented by a matrix of the form
\[
\left[
\ba{cccc}
\star & 0 & 0 & 0 \\
\star & \star & \star & \star \\
\star & \star & \star & \star \\
\star & \star & \star & \star
\ea
\right] \qquad \text{or} \qquad \left[
\ba{cccc}
\star & \star & 0 & 0 \\
\star & \star & 0 & 0 \\
\star & \star & \star & \star \\
\star & \star & \star & \star
\ea
\right] \qquad \text{or} \qquad \left[
\ba{cccc}
\star & \star & \star & 0 \\
\star & \star & \star & 0 \\
\star & \star & \star & 0 \\
\star & \star & \star & \star
\ea
\right].
\]
The sheaves in the second case are precisely the sheaves with resolution of the form
\[
\tag{ii}
0 \lra 4\O(-2) \oplus \O(-1) \stackrel{\f}{\lra} 3\O(-1) \oplus 2\O \lra \F \lra 0,
\]
where $\f_{12}=0$, $\f_{11}$ is semi-stable as a Kronecker $V$-module and
$\f_{22}$ has linearly independent entries.
\end{prop}

\begin{proof}
The first statement follows from 6.4 \cite{maican}.
The rest of the proposition follows by duality from 4.3 op.cit.
\end{proof}

\noi
Let $\W_0 = \Hom(4\O(-2), 2\O(-1) \oplus 2\O)$ and let $W_0 \subset \W_0$
be the set of morphisms $\f$ from \ref{2.1}(i).
Let
\[
G_0 = (\Aut(4\O(-2)) \times \Aut(2\O(-1) \oplus 2\O))/\C^*
\]
be the natural group acting by conjugation on $\W_0$.
Let $X_0 \subset \M(6,2)$ be the set of stable-equivalence classes of sheaves
$\F$ as in \ref{2.1}(i). This set is open and dense.

\begin{prop}
\label{2.2}
There exists a categorical quotient of $W_0$ by $G_0$ and it is isomorphic to $X_0$.
\end{prop}

\begin{proof}
We have a canonical morphism $\rho \colon W_0 \to X_0$ mapping $\f$ to the stable-equivalence
class of $\Coker(\f)$. As at 4.2.1 \cite{drezet-maican}, $\rho(\f_1) = \rho(\f_2)$ if and only if
$\overline{G \f_1} \cap \overline{G \f_2} \neq \emptyset$.
Thus any $G_0$-invariant morphism of varieties $f \colon W_0 \to Y$ factors through a unique map
$g \colon X_0 \to Y$. To show that $\rho$ is a categorical quotient map we use the method
of 3.1.6 \cite{drezet-maican}. For any sheaf $\F$ giving a point in $X_0$ we need to obtain
resolution 2.1(i) in a natural manner from the Beilinson spectral sequence converging to $\F$.
We prefer, instead, to work with the Beilinson sequence of the dual sheaf $\G= \F^\D(1)$,
which gives a point in $\M(6,4)$.
Diagram (2.2.3) \cite{drezet-maican} takes the form
\[
\xymatrix
{
2\O(-2) & 0 & 0 \\
0 & 2\O(-1) \ar[r]^-{\f_4} & 4\O
}.
\]
The exact sequence (2.2.5) \cite{drezet-maican} takes the form
\[
0 \lra 2\O(-2) \stackrel{\f_5}{\lra} \Coker(\f_4) \lra \G \lra 0.
\]
According to (2.2.4) \cite{drezet-maican}, $\f_4$ is injective.
We now easily get the exact sequence dual to 2.1(i):
\[
0 \lra 2\O(-2) \oplus 2\O(-1) \lra 4\O \lra \G \lra 0.
\qedhere
\]
\end{proof}

\begin{prop}
\label{2.3}
If $\F$ is properly semi-stable and $\PP_{\F}(t)= 6t+2$, then $\F$ gives a point in $X_0$.
\end{prop}


\section{The codimension $3$ stratum}

\noi
Let $\W_1 = \Hom(4\O(-2) \oplus \O(-1), 3\O(-1) \oplus 2\O)$ and let $W_1 \subset \W_1$
be the set of morphisms $\f$ from \ref{2.1}(ii). Let
\[
G_1 = (\Aut(4\O(-2) \oplus \O(-1)) \times \Aut(3\O(-1) \oplus 2\O))/\C^*
\]
be the natural group acting by conjugation on $\W_1$.
Let $X_1 \subset \M(6,2)$ be the set of stable-equivalence classes
of sheaves $\F$ as in \ref{2.1}(ii).

\begin{prop}
\label{3.1}
There exists a geometric quotient $W_1/G_1$ and it is a proper open subset
inside a fibre bundle over $\N(3,4,3) \times \P^2$ with fibre $\P^{20}$.
Moreover, $W_1/G_1$ is isomorphic to $X_1$. In particular, $X_1$ is irreducible
and has codimension $3$.
\end{prop}

\begin{proof}
The first statement can be proved identically as 2.2.2 \cite{mult_five}.
Let $W_1'$ be the locally closed subset of $\W_1$ given by the following conditions:
$\f_{12}=0$, $\f_{11}$ is semi-stable as a Kronecker $V$-module, $\f_{22}$ has linearly
independent entries. Let $\Sigma \subset W_1'$ be the $G_1$-invariant subset
given by the condition
\[
\f_{21}= \f_{22} u + v \f_{11}, \quad u \in \Hom(4\O(-2), \O(-1)), \quad v \in \Hom(3\O(-1),2\O).
\]
As at loc.cit., we can construct a vector bundle $Q$ over $\N(3,4,3) \times \P^2$ of rank $21$
such that $\P(Q)$ is a geometric quotient of $W_1' \setminus \Sigma$ modulo $G_1$.
Then $W_1/G_1$ is a proper open subset of $\P(Q)$.

Let $\F$ give a point in $X_1$ and let $\G = \F^\D(1)$.
The Beilinson tableau (2.2.3) \cite{drezet-maican} for $\G$ takes the form
\[
\xymatrix
{
2\O(-2) \ar[r]^-{\f_1} & \O(-1) & 0 \\
0 & 3\O(-1) \ar[r]^-{\f_4} & 4\O
}.
\]
As in the proof of 2.2.3 \cite{mult_five}, we have $\Ker(\f_1) \isom \O(-3)$ and an exact sequence
\[
0 \lra \O(-3) \oplus 3\O(-1) \lra 4\O \lra \G \lra \Coker(\f_1) \lra 0,
\]
yielding the resolution
\[
0 \lra \O(-3) \lra \O(-3) \oplus 2\O(-2) \oplus 3\O(-1) \lra \O(-1) \oplus 4\O \lra \G \lra 0.
\]
Since $\H^1(\G)=0$, we see that $\O(-3)$ can be canceled yielding the dual of
resolution 2.1(ii).
\end{proof}

\begin{prop}
\label{3.2}
The sheaves $\G$ from $X_1^\D$ are precisely the non-split extension sheaves of the form
\[
0 \lra \E \lra \G \lra \C_x \lra 0,
\]
where $\C_x$ is the structure sheaf of a point $x \in \P^2$, $\E$ gives a point in $\M(6,3)$
and satisfies the conditions $\h^0(\E(-1))=0$, $\h^1(\E)=1$.

The generic sheaves $\G$ from $X_1^\D$ are of the form $\O_C(3)(-P_1- \cdots - P_6 + P_7)$,
where $P_i$ are seven distinct points on a smooth sextic curve $C \subset \P^2$
and $P_1, \ldots, P_6$ are not contained in a conic curve.

By duality, the generic sheaves in $X_1$ are of the form $\O_C(1)(P_1 + \cdots + P_6 - P_7)$.
\end{prop}

\begin{proof}
Assume that $\G$ gives a point in $X_1^\D$, i.e. $\G \isom \Coker(\f^\T)$ for some morphism
$\f$ as in 2.1(ii). From the snake lemma we get an extension
\[
0 \lra \E \lra \G \lra \C_x \lra 0,
\]
where $x$ is the common zero of the entries of $\f_{22}$ and $\E$ has a resolution
\[
0 \lra \O(-3) \oplus 3\O(-1) \stackrel{\psi}{\lra} 4\O \lra \E \lra 0,
\]
$\psi_{12}= \f_{11}^\T$.
From 5.3 \cite{maican} we know that $\E$ gives a point in $\M(6,3)$ and satisfies the
cohomological conditions from the proposition.
Conversely, any such sheaf $\E$ is the cokernel of an injective morphism $\psi$ for which
$\psi_{12}$ is semi-stable as a Kronecker $V$-module.
Given a non-split extension of $\C_x$ by $\E$, we apply the horseshoe lemma to the above
resolution of $\E$ and to the standard resolution of $\C_x$ tensored with $\O(-1)$.
The map $\O(-1) \to \C_x$ lifts to a map $\O(-1) \to \G$ because $\H^1(\E(1))=0$.
We obtain a resolution
\[
0 \lra \O(-3) \lra \O(-3) \oplus 2\O(-2) \oplus 3\O(-1) \lra \O(-1) \oplus 4\O \lra \G \lra 0.
\]
Since $\Ext^1(\C_x, 4\O) =0$, we can deduce, as in the proof of 2.3.2 \cite{mult_five},
that the morphism $\O(-3) \to \O(-3)$ is non-zero.
We cancel $\O(-3)$ to get the dual to resolution 2.1(ii).

Let $X_{10} \subset X_1$ be the open subset of points given by sheaves $\F = \Coker(\f)$
for which the maximal minors of $\f_{11}$ have no common factor.
Let $X_{10}^\D \subset \M(6,4)$ be the dual subset.
According to \cite{modules-alternatives}, propositions 4.5 and 4.6, the sheaves $\Coker(\psi_{12})$,
where the maximal minors of $\psi_{12}$ have no common factor, are precisely the twisted ideal sheaves
$\I_Z(3)$, where $Z \subset \P^2$ is a zero-dimensional scheme of length $6$
not contained in a conic curve.
It follows that the sheaves $\G$ giving points in $X_{10}^\D$ are precisely the non-split extensions
of $\C_x$ by $\J_Z(3)$, where $\J_Z \subset \O_C$ is the ideal sheaf of a subscheme $Z$ as above
contained in a sextic curve $C$.
Take $C$ to be smooth and take $Z$ to be the union of six distinct points different from $x$.
Then $\G \isom \O_C(3)(-P_1 - \cdots - P_6 + x)$.
\end{proof}

\begin{prop}
\label{3.3}
Let $\F$ be a sheaf giving a point in $\M(6,2)$ and satisfying the conditions
$\h^0(\F(-1))=0$, $\h^1(\F)=1$.
Then $\h^0(\F \tensor \Om^1(1))= 1$ or $2$. The sheaves in the first case are precisely
the sheaves with resolution of the form
\[
0 \lra \O(-3) \oplus \O(-2) \oplus \O(-1) \stackrel{\f}{\lra} 3\O \lra \F \lra 0,
\]
where $\f$ is not equivalent to a morphism of any of the following forms:
\[
\f_1 = \left[
\ba{ccc}
\star & \star & \star \\
\star & \star & 0 \\
\star & \star & 0
\ea
\right], \qquad \f_2 = \left[
\ba{ccc}
\star & \star & \star \\
\star & 0 & \star \\
\star & 0 & \star
\ea
\right], \qquad \f_3 = \left[
\ba{ccc}
\star & \star & \star \\
\star & \star & \star \\
\star & 0 & 0
\ea
\right].
\]
\end{prop}

\begin{proof}
Let $\F$ give a point in $\M(6,2)$ and satisfy the cohomological conditions from the
hypothesis. Write $m= \h^0(\F \tensor \Om^1(1))$.
As in the proof of 2.1.4 \cite{mult_five}, the Beilinson free monad for $\F$ leads to a
resolution
\[
0 \lra \O(-3) \oplus \O(-2) \oplus m\O(-1) \stackrel{\f}{\lra} (m-1)\O(-1) \oplus 3\O \lra \F \lra 0
\]
in which $\f_{13}=0$.
As $\F$ maps surjectively onto $\Coker(\f_{11},\f_{12})$, we have $m \le 3$.
If $m = 3$, then $\Coker(\f_{11},\f_{12})$ has slope $-1/3$, so the semi-stability of $\F$
gets contradicted. Thus $m=1$ or $2$. Assume for the rest of this proof that $m=1$.
We have a resolution
\[
0 \lra \O(-3) \oplus \O(-2) \oplus \O(-1) \stackrel{\f}{\lra} 3\O \lra \F \lra 0.
\]
The conditions imposed on $\f$ follow from the semi-stability of $\F$.
Conversely, we assume that $\F$ has a resolution as in the proposition and we need to show that
there are no destabilising subsheaves. Assume that $\E \subset \F$ is a destabilising subsheaf.
We may take $\E$ to be semi-stable.
As $\F$ is generated by global sections, we have $\h^0(\E) < \h^0(\F)$.
Thus $\E$ gives a point in $\M(r,1)$ or $\M(r,2)$ for some $r$, $1 \le r \le 5$.
According to \ref{2.3}, the situation in which $\PP_{\E}(t)=3t+1$ is unfeasible.
Moreover, we have $\h^0(\E(-1))=0$, $\h^0(\E \tensor \Om^1(1)) \le 1$.
From the results in \cite{drezet-maican} and \cite{mult_five} we see that $\E$ may have
one of the following resolutions:
\[
\tag{1}
0 \lra \O(-1) \lra \O \lra \E \lra 0,
\]
\[
\tag{2}
0 \lra \O(-2) \lra \O \lra \E \lra 0,
\]
\[
\tag{3}
0 \lra \O(-2) \oplus \O(-1) \lra 2\O \lra \E \lra 0,
\]
\[
\tag{4}
0 \lra 2\O(-2) \lra 2\O \lra \E \lra 0,
\]
\[
\tag{5}
0 \lra 2\O(-2) \oplus \O(-1) \lra \O(-1) \oplus 2\O \lra \E \lra 0,
\]
\[
\tag{6}
0 \lra 3\O(-2) \lra \O(-1) \oplus 2\O \lra \E \lra 0,
\]
\[
\tag{7}
0 \lra 3\O(-2) \oplus \O(-1) \lra 2\O(-1) \oplus 2\O \lra \E \lra 0.
\]
Each of these resolutions must fit into a commutative diagram like diagram (*) at 3.1 \cite{mult_six_one}
in which $\a$ is injective on global sections.
For the first four resolutions $\a$ must be injective and we get the contradictory conclusions
that $\f \sim \f_1$, $\f \sim \f_2$ or $\f \sim \f_3$.
If $\E$ has resolution (5), then $\b$ cannot be injective, hence $\a$ is not injective,
hence $\Ker(\a) \isom \Ker(\b) \isom \O(-1)$ and we conclude, as in the case of resolution (4),
that $\f \sim \f_3$. If $\E$ has resolution (6), then, again, $\Ker(\a) \isom \O(-1) \isom \Ker(\b)$,
which is absurd, because $\O(-1)$ cannot be isomorphic to a subsheaf of $3\O(-2)$.
For resolution (7) we arrive at a contradiction in a similar manner.
\end{proof}

\noi
Let $\W_2 = \Hom(\O(-3) \oplus \O(-2) \oplus \O(-1), 3\O)$ and let $W_2 \subset \W_2$
be the set of morphisms $\f$ from proposition \ref{3.3}. Let
\[
G_2 = (\Aut(\O(-3) \oplus \O(-2) \oplus \O(-1)) \times \Aut(3\O))/\C^*
\]
be the natural group acting by conjugation on $\W_2$.
Let $X_2 \subset \M(6,2)$ be the set of stable-equivalence classes of sheaves
of the form $\Coker(\f)$, $\f \in W_2$.

\begin{prop}
\label{3.4}
There exists a geometric quotient $W_2/G_2$, which is isomorphic to $X_2$.
In particular, $X_2$ is irreducible and has codimension $3$.
\end{prop}

\begin{proof}
Diagram (2.2.3) \cite{drezet-maican} for a sheaf $\F$ giving a point in $X_2$ takes the form
\[
\xymatrix
{
4\O(-2) \ar[r]^-{\f_1} & 3\O(-1) \ar[r]^-{\f_2} & \O \\
0 & \O(-1) \ar[r]^-{\f_4} & 3\O
}.
\]
As in the proof of 2.2.4 \cite{mult_five}, we may assume that $\f_1$ and $\f_2$ are given by
\[
\f_1 = \left[
\ba{cccc}
-Y & -Z & \phantom{-} 0 & 0 \\
\phantom{-} X & \phantom{-} 0 & -Z & 0 \\
\phantom{-} 0 & \phantom{-} X & \phantom{-} Y & 0
\ea
\right], \qquad \qquad \f_2 = \left[
\ba{ccc}
X & Y & Z
\ea
\right].
\]
Thus $\Ker(\f_1) \isom \O(-3) \oplus \O(-2)$ and $\Im(\f_1) = \Ker(\f_2)$.
The exact sequence (2.2.5) \cite{drezet-maican} takes the form
\[
0 \lra \O(-3) \oplus \O(-2) \stackrel{\f_5}{\lra} \Coker(\f_4) \lra \F \lra 0.
\]
By 2.2 \cite{drezet-maican}, $\f_4$ is injective.
Clearly $\f_5$ lifts to a morphism $\f_5' \colon \O(-3) \oplus \O(-2) \to 3\O$.
We obtain the resolution
\[
0 \lra \O(-3) \oplus \O(-2) \oplus \O(-1) \stackrel{\f}{\lra} 3\O \lra \F \lra 0,
\]
\[
\f = \left[
\ba{cc}
\f_5' & \f_4
\ea
\right].
\]
This proves that the map $W_2 \to X_2$ is a categorical quotient.
According to \cite{mumford}, remark (2), p. 5, $X_2$ is normal.
Applying \cite{popov-vinberg}, theorem 4.2,
we conclude that the map $W_2 \to X_2$ is a geometric quotient.
\end{proof}


\section{The codimension $5$ stratum}

\begin{prop}
\label{4.1}
The sheaves $\F$ giving points in $\M(6,2)$ and satisfying the cohomological conditions
\[
\h^0(\F(-1))=0, \qquad \h^1(\F)=1, \qquad \h^0(\F \tensor \Om^1(1))=2
\]
are precisely the sheaves with resolution of the form
\[
0 \lra \O(-3) \oplus \O(-2) \oplus 2\O(-1) \stackrel{\f}{\lra} \O(-1) \oplus 3\O \lra \F \lra 0,
\]
where $\f_{12} \neq 0$, $\f_{13}=0$, $\f_{11}$ is not divisible by $\f_{12}$ and
$\f_{23}$ has linearly independent maximal minors.
\end{prop}

\begin{proof}
At 3.3 we proved that a sheaf $\F$ giving a point in $\M(6,2)$ and satisfying the above
cohomological conditions has a resolution as in the proposition.
The conditions imposed on $\f$ follow from the semi-stability of $\F$.

Conversely, assume that $\F$ has a resolution as in the proposition.
Assume that there is a destabilising subsheaf $\E \subset \F$.
We may assume that $\E$ is semi-stable.
From the snake lemma we obtain an extension
\[
0 \lra \F' \lra \F \lra \O_Z \lra 0,
\]
where $Z$ is the zero-dimensional scheme of length $2$
given by the ideal $(\f_{11}, \f_{12})$ and $\F'$ has a resolution
\[
0 \lra \O(-4) \oplus 2\O(-1) \stackrel{\psi}{\lra} 3\O \lra \F' \lra 0
\]
in which $\psi_{12}= \f_{23}$. According to 5.2 \cite{mult_six_one}, $\F'$ gives a point in
$\M(6,0)$ and the only subsheaf of $\F'$ of slope zero, if there is one, must be of the form
$\O_L(-1)$ for a certain line $L \subset \P^2$.
It follows that $\E$ must have Hilbert polynomial $\PP_{\E}(t) = 2t+1$, $t+2$ or $t+1$.
If $\PP_{\E}(t)=2t+1$, then $\E$ is the structure sheaf of some conic curve $C \subset \P^2$.
We obtain a commutative diagram with exact rows and injective vertical maps
\[
\xymatrix
{
0 \ar[r] & \O(-2) \ar[r] \ar[d]^-{\b} & \O \ar[r] \ar[d]^-{\a} & \O_C \ar[r] \ar[d] & 0 \\
0 \ar[r] & \O(-3) \oplus \O(-2) \oplus 2\O(-1) \ar[r] & \O(-1) \oplus 3\O \ar[r] & \F \ar[r] & 0
}.
\]
Taking into account the possible canonical forms for $\b$, we see that $\f$ is represented by
a matrix having one of the following forms:
\[
\left[
\ba{cccc}
\star & 0 & 0 & 0 \\
\star & 0 & \star & \star \\
\star & 0 & \star & \star \\
\star & \star & \star & \star
\ea
\right], \qquad \quad \left[
\ba{cccc}
\star & \star & 0 & 0 \\
\star & \star & \star & 0 \\
\star & \star & \star & 0 \\
\star & \star & \star & \star
\ea
\right], \qquad \quad \left[
\ba{cccc}
\star & \star & 0 & 0 \\
\star & \star & 0 & 0 \\
\star & \star & \star & \star \\
\star & \star & \star & \star
\ea
\right].
\]
In each of these situations the hypothesis on $\f$ gets contradicted.
If $\PP_{\E}(t)=t+1$, then $\E$ is the structure sheaf of some line $L \subset \P^2$
and we obtain a contradiction as above.
The case in which $\PP_{\E}(t)= t+2$ is not feasible because in this case $\E \isom \O_L(1)$,
yet $\H^0(\E(-1))$ must vanish because the corresponding group for $\F$ vanishes.
\end{proof}

\noi \\
Let $\W_3 = \Hom(\O(-3) \oplus \O(-2) \oplus 2\O(-1), \O(-1) \oplus 3\O)$ and let $W_3 \subset \W_3$
be the set of morphisms $\f$ from proposition \ref{4.1}. Let
\[
G_3 = (\Aut(\O(-3) \oplus \O(-2) \oplus 2\O(-1)) \times \Aut(\O(-1) \oplus 3\O))/\C^*
\]
be the natural group acting by conjugation on $\W_3$. Let $X_3 \subset \M(6,2)$
be the set of stable-equivalence classes of sheaves of the form $\Coker(\f)$, $\f \in W_3$.

\begin{prop}
\label{4.2}
The generic sheaves in $X_3$ have the form $\O_C(2)(-P_1-P_2-P_3+P_4+P_5)$,
where $C \subset \P^2$ is a smooth sextic curve, $P_i$ are five distinct points on $C$
and $P_1, P_2, P_3$ are non-colinear.
In particular, $X_3$ lies in the closure of $X_1$.
Moreover, $X_3$ also lies in the closure of $X_2$.
\end{prop}

\begin{proof}
Let $X_{30} \subset X_3$ be the open subset given by the following conditions:
the equation $\det(\f)=0$ determines a smooth sextic curve $C \subset \P^2$,
the scheme $Z$ from \ref{4.1} consists of two distinct points $P_4, P_5$,
the maximal minors of $\f_{23}$ have no common factor and the subscheme $Y \subset \P^2$
they determine consists of three distinct points $P_1, P_2, P_3$, which are also distinct from
$P_4$ and $P_5$.
Let $\F$ give a point in $X_{30}$.
According to 5.2 \cite{mult_six_one},
the sheaf $\F'$ from \ref{4.1} is isomorphic to $\O_C(2)(-P_1-P_2-P_3)$,
hence $\F$ is isomorphic to $\O_C(-P_1-P_2-P_3+P_4+P_5)$.
Conversely, we must show that any such sheaf $\F$ gives a point in $X_{30}$.
We claim that $\F(1)$ has a global section that does not vanish at $P_4$ or $P_5$.
The argument can be found at 2.3.2 \cite{mult_five} and it will be reproduced here
for the sake of completeness. Let $\e_i \colon \H^0(\O_Z) \to \C$ be the linear form of evaluation
at $P_i$, $i=4, 5$. Let $\d \colon \H^0(\O_Z) \to \H^1(\O_C(3)(-Y))$ be the connecting homomorphism
arising from the exact sequence
\[
0 \lra \O_C(3)(-Y) \lra \F(1) \lra \O_Z \lra 0.
\]
We must show that each $\e_i$ is not orthogonal to $\operatorname{Ker}(\d)$ or, which is the same,
that each $\e_i$ is not in the image of the dual map $\d^*$.
By Serre duality $\d^*$ is the restriction morphism
\[
\H^0(\O_C(Y)) = \H^0(\O_C(-3)(Y) \tensor \omega_C) \lra \H^0((\O_C(-3)(Y) \tensor \omega_C)|_{Z})
= \H^0(\O_C(Y)|_{Z}).
\]
We have the identity $\H^0(\O_C(Y)) = \H^0(\O_C)=\C$.
This follows from the fact that the connecting homomorphism associated to the exact sequence
\[
0 \lra \O_C \lra \O_C(Y) \lra \O_Y \lra 0
\]
is injective. By Serre duality, this is equivalent to saying that the restriction morphism
\[
\H^0(\O_C(3)) = \H^0(\O_C \tensor \omega_C) \lra \H^0((\O_C \tensor \omega_C)|_{Y}) =
\H^0(\O_C(3)|_{Y})
\]
is surjective, and this is obvious. The claim now easily follows.
We may now apply the horseshoe lemma to the extension
\[
0 \lra \O_C(2)(-Y) \lra \F \lra \O_Z \lra 0
\]
and to the resolutions
\[
0 \lra \O(-4) \lra \O(-3) \oplus \O(-2) \lra \O(-1) \lra \O_Z \lra 0,
\]
\[
0 \lra \O(-4) \lra \I_Z(2) \lra \O_C(2)(-Y) \lra 0.
\]
Here $\I_Y \subset \O_{\P^2}$ is the ideal sheaf of $Y$.
We obtain the resolution
\[
0 \lra \O(-4) \lra \O(-4) \oplus \O(-3) \oplus \O(-2) \lra \O(-1) \oplus \I_Y(2) \lra \F \lra 0.
\]
As in the proof of 2.3.2 \cite{mult_five}, we can show that the morphism $\O(-4) \to \O(-4)$
above is non-zero. The argument uses the fact that $\Ext^1(\O_Z, \I_Y(2))=0$.
The vanishing of this group follows from the vanishing of $\Hom(\O_Z,\O_Y)$
and of $\Ext^1(\O_Z, \O(2))$, in view of the long Ext-sequence associated to the exact sequence
\[
0 \lra \I_Y(2) \lra \O(2) \lra \O_Y \lra 0.
\]
Canceling $\O(-4)$ and taking into account that $\I_Y(2) \isom \Coker(\psi)$ for some morphism
$\psi \colon 2\O(-1) \to 3\O$ that is represented by a matrix with linearly independent maximal
minors generating the ideal of $Y$ (cf. the proof of 2.3.4(i) \cite{mult_five}), we obtain a resolution
\[
0 \lra \O(-3) \oplus \O(-2) \oplus 2\O(-1) \stackrel{\f}{\lra} \O(-1) \oplus 3\O \lra \F \lra 0,
\]
in which $\f_{13} = 0$, $\f_{23}=\psi$ and $\f_{11}, \f_{12}$ generate the ideal of $Z$.
It is clear now that $\F$ gives a point in $X_{30}$.

To show that $X_3$ is included in $\overline{X}_1$ we choose a point in $X_3$ represented by the sheaf
\[
\O_C(2)(-P_1-P_2-P_3+P_4+P_5).
\]
We may assume that the line through $P_1$ and $P_2$ intersects $C$ at six distinct points
$P_1, P_2, Q_1, Q_2, Q_3, Q_4$, which are also distinct from $P_4$ and $P_5$.
Then
\[
\O_C(2)(-P_1-P_2-P_3+P_4+P_5) \isom \O_C(1)(Q_1+Q_2+Q_3+Q_4-P_3+P_4+P_5).
\]
Clearly, we can find points $R_i$ on $C$ converging to $Q_i$, $1 \le i \le 4$,
which are distinct from $P_3$ and such that $R_1, R_2, R_3, R_4, P_4, P_5$
do not lie on a conic curve. According to \ref{3.2}, the sheaves
\[
\O_C(1)(R_1+R_2+R_3+R_4-P_3+P_4+P_5)
\]
represent points in $X_1$. These points converge to the chosen point in $X_3$.
Thus $X_3 \subset \overline{X}_1$.

Taking into account the description, found at 3.3, of sheaves giving points in $X_2$,
it is clear that for generic $\f \in W_3$ and for $t \in \C^*$ in a neighbourhood of zero
the morphism $\f + t\pi$ is injective and its cokernel gives a point in $X_2$.
Here $\pi$ is projection onto the last component followed by injection into the first component.
Clearly $[\Coker(\f + t\pi)]$ converges to $[\Coker(\f)]$ as $t$ tends to $0$. Thus $X_3 \subset \overline{X}_2$.
\end{proof}

\begin{prop}
\label{4.3}
There exists a geometric quotient $W_3/G_3$ and it is a proper open subset inside a fibre bundle
with fibre $\P^{22}$ and base $Y \times \N(3,2,3)$, where $Y$ is the Hilbert scheme
of zero-dimensional subschemes of $\P^2$ of length $2$.
Moreover, $W_3/G_3$ is isomorphic to $X_3$.
\end{prop}

\begin{proof}
The construction of $W_3/G_3$ is identical to the construction of the quotient at 3.2.3 \cite{mult_five}.
Let $W_3' \subset \W_3$ be the locally closed subset given by the conditions of \ref{4.1},
except injectivity. Let $\Sigma \subset W_3'$ be the $G_3$-invariant subset given by the condition
\[
\f_{21} = \f_{22} u + v \f_{11}, \quad u \in \Hom(\O(-3) \oplus \O(-2), 2\O(-1)), \quad v \in \Hom(\O(-1),3\O).
\]
As at loc.cit., we can construct a vector bundle $F$ over $Y \times \N(3,2,3)$ of rank $23$
such that $\P(F)$ is a geometric quotient of $W_3' \setminus \Sigma$ modulo $G_3$.
Then $W_3/G_3$ is a proper open subset of $\P(F)$.

Let $\F$ give a point in $X_3$. The Beilinson tableau (2.2.3) \cite{drezet-maican} for $\F$ has the form
\[
\xymatrix
{
4\O(-2) \ar[r]^-{\f_1} & 4\O(-1) \ar[r]^-{\f_2} & \O \\
0 & 2\O(-1) \ar[r]^-{\f_4} & 3\O
}.
\]
As at 6.5 \cite{mult_six_one}, we have $\Ker(\f_1) \isom \O(-4)$ and $\Ker(\f_2)/\Im(\f_1) \isom \O_Z$
for a scheme $Z \subset \P^2$ of dimension zero and length $2$.
The exact sequence (2.2.5) \cite{drezet-maican} takes the form
\[
0 \lra \O(-4) \stackrel{\f_5}{\lra} \Coker(\f_4) \lra \F \lra \O_Z \lra 0.
\]
We claim that $\F(1)$ has a global section which maps to a global section of $\O_Z$ that generates
this sheaf as an $\O_{\P^2}$-module. We have $\h^0(\Coker(\f_5)(1))=7$, $\h^0(\F(1))=8$, hence
$\F(1)$ has a global section mapping to a non-zero section $s$ of $\O_Z$.
Consider an extension
\[
0 \lra \C_{z_1} \lra \O_Z \lra \C_{z_2} \lra 0,
\]
where $z_1, z_2$ are not necessarily distinct points in $\P^2$.
If $s$ maps to zero in $\C_{z_2}$, then $s$ generates $\C_{z_1}$.
Let $\F'$ be the preimage of $\C_{z_1}$ in $\F$.
We apply the horseshoe lemma to the extension
\[
0 \lra \Coker(\f_5) \lra \F' \lra \C_{z_1} \lra 0
\]
and to the resolutions
\[
0 \lra \O(-4) \oplus 2\O(-1) \lra 3\O \lra \Coker(\f_5) \lra 0,
\]
\[
0 \lra \O(-3) \lra 2\O(-2) \lra \O(-1) \lra \C_{z_1} \lra 0.
\]
We obtain the resolution
\[
0 \lra \O(-3) \lra \O(-4) \oplus 2\O(-2) \oplus 2\O(-1) \lra \O(-1) \oplus 3\O \lra \F' \lra 0
\]
from which we get the relation $\h^0(\F')=4$. This is absurd, $\h^0(\F')$ cannot exceed $\h^0(\F)$.
Thus the image of $s$ in $\C_{z_2}$ is non-zero. When $z_1 = z_2$ this is enough to conclude that 
$s$ generates $\O_{Z}$. When $z_1 \neq z_2$ we revert the roles of $z_1$ and $z_2$ in the above
argument to deduce that $s$ also does not vanish at $z_1$, so $s$ generates $\O_Z$.
We can now apply the horseshoe lemma to the extension
\[
0 \lra \Coker(\f_5) \lra \F \lra \O_Z \lra 0,
\]
to the above resolution of $\Coker(\f_5)$ and to the resolution
\[
0 \lra \O(-4) \lra \O(-3) \oplus \O(-2) \stackrel{\psi}{\lra} \O(-1) \lra \O_Z \lra 0.
\]
We obtain a resolution of the form
\[
0 \lra \O(-4) \lra \O(-4) \oplus \O(-3) \oplus \O(-2) \oplus 2\O(-1) \lra \O(-1) \oplus 3\O \lra \F \lra 0.
\]
Since $\h^1(\F)=1$, the morphism $\O(-4) \to \O(-4)$ above is non-zero.
Canceling $\O(-4)$ we arrive at resolution \ref{4.1}. In view of the method at 3.1.6 \cite{drezet-maican},
we have proven that the canonical bijective map $W_3/G_3 \to X_3$ is an isomorphism.
\end{proof}

\begin{prop}
\label{4.4}
The sheaves $\F$ giving points in $\M(6,2)$ and satisfying the cohomological conditions
$\h^0(\F(-1))=1$, $\h^1(\F)=1$ are precisely the sheaves having resolution of the form
\[
0 \lra \O(-3) \oplus 2\O(-2) \stackrel{\f}{\lra} 2\O(-1) \oplus \O(1) \lra \F \lra 0,
\]
where $\f$ is not equivalent to a morphism represented by a matrix having one of the following
forms:
\[
\left[
\ba{ccc}
\star & 0 & 0 \\
\star & \star & \star \\
\star & \star & \star
\ea
\right], \quad \left[
\ba{ccc}
\star & \star & 0 \\
\star & \star & 0 \\
\star & \star & \star
\ea
\right], \quad \left[
\ba{ccc}
0 & 0 & \star \\
\star & \star & \star \\
\star & \star & \star
\ea
\right], \quad \left[
\ba{ccc}
0 & \star & \star \\
0 & \star & \star \\
\star & \star & \star
\ea
\right].
\]
\end{prop}

\begin{proof}
Let $\F$ give a point in $\M(6,2)$ and satisfy the cohomological conditions from the proposition.
Write $m=\h^0(\F \tensor \Om^1(1))$.
The Beilinson free monad (2.2.1) \cite{drezet-maican} for $\F$ reads
\[
0 \lra \O(-2) \lra 5\O(-2) \oplus m\O(-1) \lra (m+2)\O(-1) \oplus 3\O \lra \O \lra 0
\]
and gives the resolution
\[
0 \lra \O(-2) \lra 5\O(-2) \oplus m\O(-1) \lra \Om^1 \oplus (m-1)\O(-1) \oplus 3\O \lra \F \lra 0.
\]
Using the Euler sequence and arguing as at 2.1.4 \cite{mult_five} we arrive at a resolution
\[
0 \lra \O(-2) \stackrel{\eta}{\lra} \O(-3) \oplus 2\O(-2) \oplus m\O(-1) \stackrel{\f}{\lra} (m-1)\O(-1) \oplus 3\O
\lra \F \lra 0,
\]
\[
\eta = \left[
\ba{c}
0 \\ 0 \\ \eta_{31}
\ea
\right], \qquad \f = \left[
\ba{ccc}
\f_{11} & \f_{12} & 0 \\
\f_{21} & \f_{22} & \f_{23}
\ea
\right].
\]
As at loc.cit., the entries of $\eta_{31}$ span $V^*$, hence $m \ge 3$.
From the fact that $\F$ maps surjectively onto $\Coker(\f_{11}, \f_{12})$ we get the reverse inequality.
Thus $m=3$, $\Coker(\eta_{31}) \isom \Om^1(1)$ and we have a resolution
\[
0 \lra \O(-3) \oplus 2\O(-2) \oplus \Om^1(1) \stackrel{\f}{\lra} 2\O(-1) \oplus 3\O \lra \F \lra 0
\]
in which $\f_{13}=0$.
Arguing as at loc.cit., we can show that $\Coker(\f_{23}) \isom \O(1)$, so we arrive at a resolution as in the
proposition. The conditions imposed on $\f$ follow from the semi-stability of $\F$.

Conversely, we assume that $\F$ has a resolution as in the proposition and we must show that
there are no destabilising subsheaves. Write
\[
\psi = \left[
\ba{cc}
\f_{11} & \f_{12}
\ea
\right] = \left[
\ba{ccc}
q_1 & \ell_{11} & \ell_{12} \\
q_2 & \ell_{21} & \ell_{22}
\ea
\right].
\]
As noted at 4.1 \cite{mult_six_one}, the conditions on $\f$ in the proposition are equivalent
to saying that
\[
\left|
\ba{cc}
\ell_{11} & \ell_{12} \\
\ell_{21} & \ell_{22}
\ea
\right| \neq 0 \quad \text{and} \quad \left|
\ba{cc}
q_1 & \ell_{11} \\
q_2 & \ell_{21}
\ea
\right|, \quad \left|
\ba{cc}
q_1 & \ell_{12} \\
q_2 & \ell_{22}
\ea
\right|
\]
are linearly independent in $\SS^3 V^*/ (\ell_{11} \ell_{22} - \ell_{12} \ell_{21})V^*$.
Thus the maximal minors of $\psi$ cannot have a quadratic common factor.
It follows that $\Ker(\psi) \isom \O(-4)$, if the maximal minors of $\psi$ have a linear common factor,
or $\Ker(\psi) \isom \O(-5)$, if they have no common factor.
From the snake lemma we have an exact sequence
\[
0 \lra \Ker(\psi) \lra \O(1) \lra \F \lra \Coker(\psi) \lra 0.
\]
Assume that $\Ker(\psi) \isom \O(-4)$. Because of the conditions on $\psi$ it is easy to check that
$\Coker(\psi)$ has zero-dimensional torsion of length at most $1$.
Assume that $\Coker(\psi)$ has no zero-dimensional torsion.
Then $\Coker(\psi) \isom \O_L(1)$ for a line $L \subset \P^2$
and we have an extension
\[
0 \lra \O_C(1) \lra \F \lra \O_L(1) \lra 0,
\]
where $C \subset \P^2$ is a quintic curve.
Let $\F' \subset \F$ be a non-zero subsheaf of multiplicity at most $5$.
Denote by $\CC$ its image in $\O_L(1)$ and put $\K= \F' \cap \O_C(1)$.
Let $\A$ be a sheaf as in 3.1.2 \cite{mult_five}. If $\CC=0$, then $\pp(\F') \le 0$ because $\O_C(1)$ is stable.
We may, therefore, assume that $\CC \neq 0$.
We can estimate the slope of $\F'$ as at loc.cit.:
\begin{align*}
\PP_{\F'}(t) & = \PP_{\K}(t) + \PP_{\CC}(t) \\
& = \PP_{\A}(t) - \h^0(\A/\K) + \PP_{\O_L(1)}(t) - \h^0(\O_L(1)/\CC) \\
& = (5-d)t + \frac{d^2-5d}{2} + t + 2 - \h^0(\A/\K) - \h^0(\O_L(1)/\CC),
\end{align*}
where $d$ is an integer, $1 \le d \le 4$. Thus
\[
\pp(\F')= \frac{1}{6-d}\left( \frac{d^2-5d}{2} + 2 - \h^0(\A/\K) - \h^0(\O_L(1)/\CC) \right) \le
\frac{d^2-5d+4}{2(6-d)} < \frac{1}{3} = \pp(\F).
\]
We see that in this case $\F$ is stable.
Assume next that $\Coker(\psi)$ has a zero-dimensional subsheaf $\TT$ of length $1$.
Let $\E$ be the preimage of $\TT$ in $\F$.
According to 3.1.5 \cite{mult_five}, $\E$ gives a point in $\M(5,1)$.
Let $\F'$ and $\CC$ be as above. If $\CC \subset \TT$, then $\F' \subset \E$, hence
$\pp(\F') \le \pp(\E) < \pp(\F)$. If $\CC$ is not a subsheaf of $\TT$, then we can estimate the slope
of $\F'$ as above concluding again that it is less than the slope of $\F$.

Assume now that $\Ker(\psi) \isom \O(-5)$. We have an extension
\[
0 \lra \O_C(1) \lra \F \lra \TT \lra 0,
\]
where $C \subset \P^2$ is a sextic curve and $\TT$ is a zero-dimensional sheaf of length $5$.
Let $\F' \subset \F$ be a subsheaf of multiplicity at most $5$, let $\TT'$ be its image in $\TT$
and put $\K = \F' \cap \O_C(1)$. As above, we have
\begin{align*}
\PP_{\F'}(t) & = \PP_{\K}(t) + \h^0(\TT') \\
& = \PP_{\A}(t) - \h^0(\A/\K) + \h^0(\TT') \\
& = (6-d)t + \frac{d^2-5d-6}{2} - \h^0(\A/\K) + \h^0(\TT'), \\
\pp(\F') & = - \frac{d+1}{2} + \frac{\h^0(\TT') - \h^0(\A/\K)}{6-d} \le - \frac{d+1}{2} + \frac{5}{6-d},
\end{align*}
where $d$ is an integer, $1 \le d \le 5$. We see from this that $\pp(\F') < \pp(\F)$
except, possibly, when $d=5$ and $\PP_{\F'}(t)= t+1$ or $t+2$, i.e. when $\F'$ is isomorphic to
$\O_L$ or $\O_L(1)$ for a line $L \subset \P^2$.
These situations can easily be ruled out.
If, say, $\O_L$ were a subsheaf of $\F$, then we would get a commutative diagram
\[
\xymatrix
{
0 \ar[r] & \O(-1) \ar[r] \ar[d]^-{\b} & \O \ar[r] \ar[d]^-{\a} & \O_L \ar[r] \ar[d] & 0 \\
0 \ar[r] & \O(-3) \oplus 2\O(-2) \ar[r] & 2\O(-1) \oplus \O(1) \ar[r] & \F \ar[r] & 0
}
\]
in which $\a$ is injective, because it is injective on global sections.
Thus $\b$ is also injective, which is absurd.
We conclude that $\F$ is stable.
\end{proof}

\noi
Let $\W_4= \Hom(\O(-3) \oplus 2\O(-2), 2\O(-1) \oplus \O(1))$ and let $W_4 \subset \W_4$
be the set of morphisms $\f$ from proposition \ref{4.4}.
Let
\[
G_4 = (\Aut(\O(-3) \oplus 2\O(-2)) \times \Aut(2\O(-1) \oplus \O(1)))/\C^*
\]
be the natural group acting by conjugation on $\W_4$. Let $X_4 \subset \M(6,2)$
be the set of stable-equivalence classes of sheaves of the form $\Coker(\f)$, $\f \in W_4$.

\begin{prop}
\label{4.5}
There exists a geometric quotient $W_4/G_4$, which is isomorphic to $X_4$.
In particular, $X_4$ is irreducible and has codimension $5$.
\end{prop}

\begin{proof}
The Beilinson diagram (2.2.3) \cite{drezet-maican} for the dual sheaf $\G= \F^\D(1)$ giving
a point in $\M(6,4)$ has the form
\[
\xymatrix
{
3\O(-2) \ar[r]^-{\f_1} & 3\O(-1) \ar[r]^-{\f_2} & \O \\
\O(-2) \ar[r]^-{\f_3} & 5\O(-1) \ar[r]^-{\f_4} & 5\O
}.
\]
As in the proof of 2.2.4 \cite{mult_five}, we have $\Ker(\f_2)= \Im(\f_1)$ and $\Ker(\f_1) \isom \O(-3)$.
Combining the exact sequences (2.2.4) and (2.2.5) \cite{drezet-maican} we get the resolution
\[
0 \lra \O(-2) \stackrel{\psi}{\lra} \O(-3) \oplus 5\O(-1) \lra 5\O \lra \G \lra 0.
\]
As in the proof of 2.1.4 \cite{mult_five}, we have $\Coker(\psi) \isom \O(-3) \oplus 2\O(-1) \oplus \Om^1(1)$.
We get the resolution
\[
0 \lra \O(-3) \oplus 2\O(-1) \oplus \Om^1(1) \stackrel{\f}{\lra} 5\O \lra \G \lra 0.
\]
As at loc.cit., we have $\Coker(\f_{13}) \isom 2\O \oplus \O(1)$. We finally arrive at the resolution
dual to resolution \ref{4.4}:
\[
0 \lra \O(-3) \oplus 2\O(-1) \lra 2\O \oplus \O(1) \lra \G \lra 0.
\]
This proves that the map $W_4 \to X_4$ is a categorical quotient.
According to \cite{mumford}, remark (2), p. 5, $X_4$ is normal.
Applying \cite{popov-vinberg}, theorem 4.2,
we conclude that the map $W_4 \to X_4$ is a geometric quotient.
\end{proof}

\begin{prop}
\label{4.6}
The generic sheaves in $X_4$ are of the form $\O_C(1)(P_1 + \cdots + P_5)$, where $C \subset \P^2$
is a smooth sextic curve and $P_i$ are five distinct points on $C$, no three of which are colinear.
In particular, $X_4$ lies in the closure of $X_1$.
\end{prop}

\begin{proof}
Let $X_{40} \subset X_4$ be the subset defined by the following conditions:
the sextic curve $C$ given by the equation $\det(\f) = 0$ is smooth, the conic curve $F$
given by the equation $f= \ell_{11} \ell_{22} - \ell_{12} \ell_{21} =0$ is irreducible,
there are constants $c_1, c_2 \in \C$ such that the cubic curve $G$ with equation
\[
c_1 \left|
\ba{cc}
q_1 & \ell_{11} \\
q_2 & \ell_{21}
\ea
\right| + c_2 \left|
\ba{cc}
q_1 & \ell_{12} \\
q_2 & \ell_{22}
\ea
\right| = 0
\]
meets $F$ at six distinct points $P_1, \ldots, P_6$ (notations as at \ref{4.4}).
Let $\F = \Coker(\f)$ give a point in $X_{40}$.
Performing, possibly, column operations on the matrix representing $\f$
we may assume that $c_1 = 0, c_2=1$ and that $P_6$ is given by the equations
$\ell_{12}=0$, $\ell_{22}=0$. Then $\Coker(\psi) \isom \O_Z$, where $Z$ is the union of
$P_1, \ldots, P_5$.
As at \ref{4.4}, $\F$ is an extension of $\O_Z$ by $\O_C(1)$, hence $\F \isom \O_C(1)(P_1 + \cdots
+ P_5)$. Since $P_1, \ldots, P_5$ are on the irreducible conic $F$, no three of them are colinear.

Conversely, we must show that every sheaf of the form $\O_C(1)(P_1+ \cdots +P_5)$
gives a point in $X_{40}$.
Let $F \subset \P^2$ be a conic curve containing $P_1, \ldots, P_5$.
Because these points are assumed to be in general linear position, $F$ is irreducible.
Choose a sixth point $P_6 \in F$ distinct from the others.
Let $G \subset \P^2$ be a cubic curve meeting $F$ precisely at $P_1, \ldots, P_6$
(for example, the union of the three lines $P_1 P_2, P_3 P_4, P_5 P_6$).
Choose equations $f=0$, $g=0$ for $F$, $G$.
Choose equations $\ell_{12} = 0$, $\ell_{22} = 0$ for $P_6$.
We may write $f= \ell_{11} \ell_{22} - \ell_{12} \ell_{21}$, $g = q_1 \ell_{22} - q_2 \ell_{12}$
for some $\ell_{11}, \ell_{21} \in V^*$ and $q_1, q_2 \in \SS^2 V^*$.
Let $\psi \colon \O(-3) \oplus 2\O(-2) \to 2\O(-1)$ be the morphism represented
by the matrix
\[
\left[
\ba{ccc}
q_1 & \ell_{11} & \ell_{12} \\
q_2 & \ell_{21} & \ell_{22}
\ea
\right].
\]
We have $\Coker(\psi) \isom \O_Z$, where $Z$ is the union of $P_1, \ldots, P_5$.
By construction, the maximal minors of $\psi$ have no common factor, hence
$\Ker(\psi) \isom \O(-5)$.
We apply the horseshoe lemma to the extension
\[
0 \lra \O_C(1) \lra \F \lra \O_Z \lra 0,
\]
to the standard resolution of $\O_C(1)$ and to the resolution
\[
0 \lra \O(-5) \lra \O(-3) \oplus 2\O(-2) \stackrel{\psi}{\lra} 2\O(-1) \lra \O_Z \lra 0.
\]
We claim that the morphism $2\O(-1) \to \O_Z$ lifts to a morphism $2\O(-1) \to \F$.
To see this let $\a \colon \H^0(2\O) \to \H^0(\O_Z)$ be the induced morphism and let
$\d \colon \H^0(\O_Z) \to \H^1(\O_C(2))$ be the connecting homomorphism associated
to the exact sequence
\[
0 \lra \O_C(2) \lra \O_C(2)(P_1 + \cdots + P_5) \lra \O_Z \lra 0.
\]
We must show that $\d \circ \a = 0$. We will show that $\a^* \circ \d^* =0$.
Taking duals in the above resolution of $\O_Z$ we obtain the resolution
\[
0 \lra 2\O(-3) \lra 2\O(-2) \oplus \O(-1) \lra \O(1) \lra {\mathcal Ext}^2(\O_Z, \omega_{\P^2}) \lra 0.
\]
The induced map on global sections
\[
V^* \isom \H^0(\O(1)) \lra \H^0({\mathcal Ext}^2(\O_Z, \omega_{\P^2})) \isom \Ext^2(\O_Z, \omega_{\P^2})
\isom \H^0(\O_Z)^*
\]
can be identified with $\d^*$ because, by Serre duality, $\d^*$ is the restriction homomorphism
\[
V^* \isom \H^0(\O_C(1)) \isom \H^0(\O_C(2)^* \tensor \omega_C) \isom \H^1(\O_C(2))^* \lra \H^0(\O_Z)^*.
\]
The induced map
\[
\Ext^2(\O_Z, \omega_{\P^2}) \isom
\H^0({\mathcal Ext}^2(\O_Z, \omega_{\P^2})) \lra \H^2(2\O(-3)) \isom \Ext^2(2\O, \omega_{\P^2})
\]
can be identified with $\a^*$. It is clear now that we have $\a^* \circ \d^* =0$, proving the claim.
We obtain the resolution
\[
0 \lra \O(-5) \lra \O(-5) \oplus \O(-3) \oplus 2\O(-2) \lra 2\O(-1) \oplus \O(1) \lra \F \lra 0.
\]
Since $\Ext^1(\O_Z, \O(1))=0$, the argument at 2.3.2 \cite{mult_five} applies to show that
the morphism $\O(-5) \to \O(-5)$ above is non-zero.
Canceling $\O(-5)$ we obtain a resolution that places $\F$ in $X_{40}$.

The inclusion $X_4 \subset \overline{X}_1$ follows from the fact that any sheaf of the form
$\O_C(1)(P_1 + \cdots + P_5)$ as above is the limit of a sequence of sheaves
$\O_C(1)(P_1 + \cdots + P_6 - P_7)$ as at \ref{3.2}.
\end{proof}

\begin{prop}
\label{4.7}
$X_4$ lies in the closure of $X_2$.
\end{prop}

\begin{proof}
The argument can be found at 2.1.6 \cite{mult_five} and can be traced back to 3.2.3 \cite{drezet-maican}.
Let $Y \subset \M(6,4)$ be the subset
of stable-equivalence classes of sheaves $\G$ satisfying the conditions $\h^0(\G(-1))=1$,
$\h^0(\G(-2))=0$. We claim that for any such sheaf we have the relation
$\h^0(\G \tensor \Om^1)=0$.
To see this denote $m = \h^0(\G \tensor \Om^1)$ and consider the Beilinson diagram
(2.2.3) \cite{drezet-maican} for $\G(-1)$:
\[
\xymatrix
{
8\O(-2) \ar[r]^-{\f_1} & (m+10)\O(-1) \ar[r]^-{\f_2} & 3\O \\
0 & m\O(-1) \ar[r]^-{\f_4} & \O
}.
\]
As $\f_4$ is injective, we have $m=0$ or $1$. If $m=1$, then $\Coker(\f_4) \isom \O_L$
for a line $L \subset \P^2$.
The exact sequence (2.2.5) \cite{drezet-maican} reads
\[
0 \lra \Ker(\f_1) \lra \O_L \lra \G(-1) \lra \Ker(\f_2)/\Im(\f_1) \lra 0.
\]
The map $\O_L \to \G(-1)$ is zero because $\pp(\O_L) > \pp(\G(-1))$ and both sheaves are semi-stable.
Thus $\O_L \isom \Ker(\f_1)$, which is absurd.

Using the Beilinson monad for $\G(-1)$ we see that $Y$ is parametrised by an open subset $M$
inside the space of monads
\[
0 \lra 8\O(-1) \stackrel{A}{\lra} 10\O \oplus \O(1) \stackrel{B}{\lra} 3\O(1) \lra 0
\]
satisfying $B_{12}=0$.
Consider the map $\Phi \colon M \to \Hom(10\O, 3\O(1))$ defined by $\Phi(A,B)= B_{11}$.
Using the vanishing of $\H^1(\G(1))$ for an arbitrary sheaf $\G$ giving a point in $Y$
(cf. 2.1.3 \cite{drezet-maican}), we can prove that $M$ is smooth and that $\Phi$ has surjective differential
at every point.
This further leads to the conclusion that the set of monads in $M$ whose cohomology
sheaf $\G$ satisfies the relation $\h^1(\G)=1$ is included in the closure of the set of monads
for which $\h^1(\G)=0$.
Thus $X_4^\D$ lies in the relative closure of $X_2^\D \cup X_3^\D$ in $Y$.
It follows that $X_4 \subset \overline{X_2 \cup X_3}$.
Since $X_3 \subset \overline{X}_2$, the conclusion follows.
\end{proof}


\section{The codimension $7$ stratum}

\begin{prop}
\label{5.1}
The sheaves $\F$ giving points in $\M(6,2)$ and satisfying the conditions
$\h^0(\F(-1))=1$, $\h^1(\F)=2$ are precisely the sheaves having resolution of the form
\[
0 \lra 2\O(-3) \oplus \O(-1) \stackrel{\f}{\lra} \O(-2) \oplus \O \oplus \O(1) \lra \F \lra 0,
\]
where $\f_{11}$ has linearly independent entries, $\f_{22} \neq 0$ and does not divide $\f_{32}$.
\end{prop}

\begin{proof}
Let $\F$ give a point in $\M(6,2)$ and satisfy the cohomological conditions from above.
Put $m= \h^0(\F \tensor \Om^1(1))$. Let $\G = \F^\D(1)$. The Beilinson monad for $\G$ gives the
resolution
\[
0 \lra 2\O(-2) \lra 4\O(-2) \oplus (m+2)\O(-1) \lra \Om^1 \oplus (m-3)\O(-1) \oplus 5\O \lra \G \lra 0.
\]
Using the Euler sequence and arguing as at 2.1.4 \cite{mult_five} we arrive at a resolution
\[
0 \lra 2\O(-2) \stackrel{\psi}{\lra} \O(-3) \oplus \O(-2) \oplus (m+2)\O(-1) \stackrel{\f}{\lra}
(m-3)\O(-1) \oplus 5\O \lra \G \lra 0,
\]
\[
\psi = \left[
\ba{c}
0 \\ 0 \\ \psi_{31}
\ea
\right], \qquad \f= \left[
\ba{ccc}
\f_{11} & \f_{12} & 0 \\
\f_{21} & \f_{22} & \f_{23}
\ea
\right].
\]
Arguing as in the proof of 3.2.5 \cite{mult_five}, we see that, modulo operations on rows and columns,
$\psi_{31}$ is represented by a matrix of the form
\[
\left[
\ba{ccccccc}
X & Y & Z & 0 & 0 & 0 & \cdots \\
0 & 0 & 0 & X & Y & Z & \cdots
\ea
\right]^\T.
\]
Thus $m \ge 4$. From the fact that $\G$ maps surjectively onto $\Coker(\f_{11}, \f_{12})$
we get the reverse inequality. Thus $m=4$, $\Coker(\psi_{31}) \isom 2\Om^1(1)$
and we obtain a resolution
\[
0 \lra \O(-3) \oplus \O(-2) \oplus 2\Om^1(1) \stackrel{\f}{\lra} \O(-1) \oplus 5\O \lra \G \lra 0,
\]
in which $\f_{13}=0$. Dually, we have the resolution
\[
0 \lra 5\O(-2) \oplus \O(-1) \lra 2\Om^1 \oplus \O \oplus \O(1) \lra \F \lra 0.
\]
Combining with the standard resolution of $\Om^1$ yields the exact sequence
\[
0 \lra 2\O(-3) \oplus 5\O(-2) \oplus \O(-1) \stackrel{\f}{\lra} 6\O(-2) \oplus \O \oplus \O(1) \lra \F \lra 0.
\]
From the semi-stability of $\F$ we see that $\rank(\f_{12})=5$, cf. argument at 2.1.4 \cite{mult_five}.
Canceling $5\O(-2)$ we obtain the desired resolution of $\F$. The conditions imposed on $\f$
follow from the semi-stability of $\F$.

Conversely, we assume that $\F$ has a resolution as in the proposition and we must show
that there are no destabilising subsheaves.
From the snake lemma we get an extension
\[
0 \lra \J_Z(2) \lra \F \lra \C_x \lra 0,
\]
where $\J_Z \subset \O_C$ is the ideal sheaf of a zero-dimensional subscheme $Z$ of length $2$
inside a sextic curve $C$ and $\C_x$ is the structure sheaf of a point.
Let $\F' \subset \F$ be a subsheaf of multiplicity at most $5$,
let $\CC$ be its image in $\C_x$ and $\K = \F' \cap \J_Z(2)$.
With the notations of \ref{4.4} we have
\begin{align*}
\PP_{\F'}(t) & = \PP_{\K}(t) + \h^0(\CC) \\
& = \PP_{\A}(t) - \h^0(\A/\K) + \h^0(\CC) \\
& = (6-d)t + \frac{d^2 -7d+6}{2} - \h^0(\A/\K) + \h^0(\CC)
\end{align*}
for some integer $d$, $1 \le d \le 5$, hence
\[
\pp(\F') = \frac{1-d}{2} + \frac{\h^0(\CC)-\h^0(\A/\K)}{6-d} \le \frac{1-d}{2} + \frac{1}{6-d} < \pp(\F).
\]
We conclude that $\F$ is stable.
\end{proof}

\noi \\
Let $\W_5 = \Hom(2\O(-3) \oplus \O(-1), \O(-2) \oplus \O \oplus \O(1))$
and let $W_5 \subset \W_5$ be the set of morphisms $\f$ from proposition \ref{5.1}.
Let
\[
G_5 = (\Aut(2\O(-3) \oplus \O(-1)) \times \Aut(\O(-2) \oplus \O \oplus \O(1)))/\C^*
\]
be the natural group acting by conjugation on $\W_5$.
Let $X_5 \subset \M(6,2)$ be the set of stable-equivalence classes of sheaves of the form
$\Coker(\f)$, $\f \in W_5$.

\begin{prop}
\label{5.2}
There exists a geometric quotient of $W_5$ by $G_5$ and it is isomorphic to a proper open
subset inside a fibre bundle with fibre $\P^{24}$ and base $\P^2 \times Y$, where $Y$
is the Hilbert scheme of zero-dimensional subschemes of $\P^2$ of length $2$.
Moreover, $W_5/G_5$ is isomorphic to $X_5$.
\end{prop}

\begin{proof}
The construction of $W_5/G_5$ is entirely analogous to the construction of the quotient at
3.2.3 \cite{mult_five}.

Let $\F$ give a point in $X_5$ and let $\G = \F^\D(1)$.
The Beilinson tableau (2.2.3) \cite{drezet-maican} for $\G$ takes the form
\[
\xymatrix
{
4\O(-2) \ar[r]^-{\f_1} & 4\O(-1) \ar[r]^-{\f_2} & \O \\
2\O(-2) \ar[r]^-{\f_3} & 6\O(-1) \ar[r]^-{\f_4} & 5\O
}.
\]
As at 6.5 \cite{mult_six_one}, we have $\Ker(\f_1) \isom \O(-4)$ and $\Ker(\f_2)/\Im(\f_1) \isom \O_Z$
for a scheme $Z \subset \P^2$ of dimension zero and length $2$.
The exact sequence (2.2.5) \cite{drezet-maican} takes the form
\[
0 \lra \O(-4) \stackrel{\f_5}{\lra} \Coker(\f_4) \lra \G \lra \O_Z \lra 0.
\]
As at 3.2.5 \cite{mult_five}, we have $\Coker(\f_3) \isom 2\Om^1(1)$.
This, together with (2.2.4) \cite{drezet-maican}, gives the resolution
\[
0 \lra \O(-4) \oplus 2\Om^1(1) \lra 5\O \lra \Coker(\f_5) \lra 0.
\]
We claim that $\G(1)$ has a global section which maps to a global section of $\O_Z$
that generates this sheaf as an $\O_{\P^2}$-module.
To show this we argue as at \ref{4.3}.
By 2.1.3 \cite{drezet-maican}, the group $\H^1(\G(1))$ vanishes, hence we have $\h^0(\G(1))=10$.
Since $\h^0(\Coker(\f_5)(1))=9$, we see that $\G(1)$ has a global section mapping to a non-zero
section $s$ of $\O_Z$.
We have an exact sequence
\[
0 \lra \C_{z_1} \lra \O_Z \lra \C_{z_2} \lra 0,
\]
where $z_1, z_2$ are not necessarily distinct points in $\P^2$.
If $s$ maps to zero in $\C_{z_2}$, then $s$ generates $\C_{z_1}$.
Let $\F'$ be the preimage of $\C_{z_1}$ in $\F$.
We can apply the horseshoe lemma to the extension
\[
0 \lra \Coker(\f_5) \lra \F' \lra \C_{z_1} \lra 0,
\]
to the above resolution of $\Coker(\f_5)$ and to the standard resolution of $\C_{z_1}$
tensored with $\O(-1)$. We obtain the exact sequence
\[
0 \lra \O(-3) \stackrel{\psi}{\lra} \O(-4) \oplus 2\O(-2) \oplus 2\Om^1(1) \lra \O(-1) \oplus 5\O \lra \F' \lra 0.
\]
We have $\h^1(\F')=\h^2(\Coker(\psi))=3$, hence $\h^0(\F')=4$.
On the other hand, $\H^0(\Coker(\psi))$ vanishes, hence $\h^0(\F') \ge 5$.
This is absurd, so an exact sequence as above cannot exist.
Thus the image of $s$ in $\C_{z_2}$ is non-zero and the claim follows as at \ref{4.3}.
We can now combine the resolutions of $\O_Z$ and of $\Coker(\f_5)$ from above
to get the exact sequence
\[
0 \lra \O(-4) \lra \O(-4) \oplus \O(-3) \oplus \O(-2) \oplus 2\Om^1(1) \lra \O(-1) \oplus 5\O \lra \G \lra 0.
\]
The map $\O(-4) \to \O(-4)$ is non-zero because $\h^1(\G)=1$.
We may cancel $\O(-4)$ to get the resolution
\[
0 \lra \O(-3) \oplus \O(-2) \oplus 2\Om^1(1) \lra \O(-1) \oplus 5\O \lra \G \lra 0.
\]
We saw at \ref{5.1} how this leads to a morphism $\f \in W_5$ such that $\F \isom \Coker(\f)$.
We conclude, as at 3.1.6 \cite{drezet-maican}, that the canonical bijective map
$W_5/G_5 \to X_5$ is an isomorphism.
\end{proof}

\begin{prop}
\label{5.3}
The generic sheaves from $X_5$ are precisely the non-split extension sheaves
\[
0 \lra \J_Z(2) \lra \F \lra \C_x \lra 0,
\]
where $\J_Z \subset \O_C$ is the ideal sheaf of a zero-dimensional scheme $Z$
of length $2$ inside a sextic curve $C \subset \P^2$ and $\C_x$ is the structure sheaf
of a point $x \in \P^2$ that is not in the support of $Z$.

There is a dense open subset of $X_5$ consisting of the isomorphism classes
of all sheaves of the form $\O_C(2)(P_1 - P_2 - P_3)$, where $C \subset \P^2$
is a smooth sextic curve and $P_1, P_2, P_3$ are distinct points on $C$.
In particular, $X_5$ lies in the closure of $X_3$ and also in the closure of $X_4$.
\end{prop}

\begin{proof}
Let $\F= \Coker(\f)$ give a point in $X_5$, where $\f$ is a morphism as at \ref{5.1}.
Let $x \in \P^2$ be the point given by the ideal generated by the entries of $\f_{11}$,
let $Z \subset \P^2$ be the subscheme given by the equations
$\f_{22}=0$, $\f_{32}=0$ and let $C \subset \P^2$ be the curve given by the equation
$\det(\f)=0$. We saw at \ref{5.1} that $\F$ is a non-split extension of $\C_x$ by $\J_Z(2)$.
Let $X_{50} \subset X_5$ be the open subset given by the condition that $x$ be not a
subscheme of $Z$. To show that every extension as in the proposition gives a point in $X_{50}$
we combine the resolutions
\[
0 \lra \O(-4) \lra 2\O(-3) \lra \O(-2) \lra \C_x \lra 0
\]
and
\[
0 \lra \O(-4) \lra \I_Z(2) \lra \J_Z(2) \lra 0.
\]
Here $\I_Z \subset \O_{\P^2}$ is the ideal sheaf of $Z$.
We obtain the resolution
\[
0 \lra \O(-4) \lra \O(-4) \oplus 2\O(-3) \lra \O(-2) \oplus \I_Z(2) \lra \F \lra 0.
\]
The group $\Ext^1(\C_x, \I_Z(2))$ vanishes because $x$ is not in $Z$, so we can apply the argument at
2.3.2 \cite{mult_five} to deduce that the morphism $\O(-4) \to \O(-4)$ in the above complex is non-zero.
Canceling $\O(-4)$ we obtain the resolution
\[
0 \lra 2\O(-3) \lra \O(-2) \oplus \I_Z(2) \lra \F \lra 0,
\]
which shows that $\F$ gives a point in $X_{50}$.

Clearly every sheaf of the form $\O_C(2)(P_1 - P_2 - P_3)$ is the limit of a sequence of sheaves
of the form $\O_C(2)(- Q_1 - Q_2 - Q_3 + Q_4 + Q_5)$ as at \ref{4.2} (make $Q_1$ converge to $Q_4$).
Thus $X_5 \subset \overline{X}_3$. To prove that $X_5 \subset \overline{X}_4$ fix a sheaf
$\F = \O_C(2)(Q_1 - P_2 - P_3)$ in $X_5$.
Choosing $\F$ general enough, we may assume that $Q_1, P_2, P_3$ are non-colinear
and that the line $P_2P_3$ meets $C$ at six distinct points $P_2, P_3, Q_2, Q_3, Q_4, Q_5$.
Then
\[
\O_C(2)(Q_1- P_2 - P_3) \isom \O_C(1)(Q_1 + \cdots + Q_5).
\]
Clearly, we can choose five distinct points $R_i$ on $C$ converging to $Q_i$, $1 \le i \le 5$,
such that no three among them are colinear.
According to proposition \ref{4.6}, $\O_C(1)(R_1 + \cdots + R_5)$ gives a point in $X_4$.
Thus $\O_C(2)(Q_1 -P_2 -P_3)$ lies in the closure of $X_4$.
\end{proof}


\section{The codimension $9$ stratum}

\begin{prop}
\label{6.1}
The sheaves $\F$ in $\M(6,2)$ satisfying the condition $\h^1(\F(1))>0$ are precisely
the sheaves with resolution of the form
\[
0 \lra \O(-4) \oplus \O \stackrel{\f}{\lra} 2\O(1) \lra \F \lra 0,
\]
\[
\f= \left[
\ba{cc}
f_1 & \ell_1 \\
f_2 & \ell_2
\ea
\right],
\]
where $\ell_1, \ell_2$ are linearly independent one-forms.
These sheaves are precisely the sheaves $\J_x(2)$, where $\J_x \subset \O_C$
is the ideal sheaf of a point $x$ on a sextic curve $C \subset \P^2$.
\end{prop}

\begin{proof}
This statement follows by duality from \cite{mult_six_one}, proposition 6.1.
\end{proof}

\noi
Let $\W_6 = \Hom(\O(-4)\oplus \O, 2\O(1))$ and let $W_6 \subset \W_6$ be the set of morphisms
$\f$ from \ref{6.1}. Let
\[
G_6 = (\Aut(\O(-4) \oplus \O) \times \Aut(2\O(1)))/\C^*
\]
be the natural group acting by conjugation on $\W_6$.
Let $X_6 \subset \M(6,2)$ denote the set of stable-equivalence classes of sheaves of the form
$\Coker(\f)$, $\f \in W_6$.

\begin{prop}
\label{6.2}
There exists a geometric quotient $W_6/G_6$, which is isomorphic to the universal sextic
$\Sigma \subset \P^2 \times \P(\SS^6 V^*)$. Moreover, $W_6/G_6$ is isomorphic to $X_6$,
so this is a smooth closed subvariety of $\M(6,2)$ of codimension $9$.
\end{prop}

\begin{proof}
For the first part of the proposition we notice, as at 3.2 \cite{drezet-maican} or at 3.2.5 \cite{mult_five},
that the map $W_6 \to \Sigma$ defined by
\[
\left[
\ba{cc}
f_1 & \ell_1 \\
f_2 & \ell_2
\ea
\right] \lra (x, \langle f_1 \ell_2 - f_2 \ell_1 \rangle),
\]
$x$ being given by the equations $\ell_1 = 0$, $\ell_2 = 0$, is a geometric quotient map.
The canonical morphism $\rho \colon W_6 \to X_6$, $\rho(\f) = [ \Coker(\f) ]$, determines a bijective
morphism
\[
\upsilon \colon \Sigma \lra X_5, \qquad \upsilon(x, \langle f \rangle) = [\J_x(2)],
\]
where $\J_x \subset \O_C$ is the ideal sheaf of $x$ on the curve $C$ given by the equation
$f=0$.
As at 6.5 \cite{mult_six_one}, in order to show that $\upsilon^{-1}$ is a morphism,
we need to construct the pair $(x,C)$ starting from $\EE^1(\J_x(2))$.
For technical reasons we will work, instead, with $\EE^1(\J_x^\D)$.
Denote $\G=\J_x^\D$ and notice that $\G$ gives a point in $\M(6,10)$ and is an extension of the form
\[
0 \lra \O_C(3) \lra \G \lra \C_x \lra 0.
\]
The Beilinson tableau (2.2.3) \cite{drezet-maican} for $\G$ has the form
\[
\xymatrix
{
2\O(-2) \ar[r]^-{\f_1} & \O(-1) & 0 \\
6\O(-2) \ar[r]^-{\f_3} & 15\O(-1) \ar[r]^-{\f_4} & 10 \O
}.
\]
Since $\G$ is semi-stable and maps surjectively onto $\Coker(\f_1)$ we see that
$\Coker(\f_1) \isom \C_y$ for a point $y \in \P^2$ and $\Ker(\f_1) \isom \O(-3)$.
The exact sequence (2.2.5) \cite{drezet-maican} reads
\[
0 \lra \O(-4) \stackrel{\f_5} \lra \Coker(\f_4) \lra \G \lra \C_y \lra 0.
\]
Denote $\G'= \Coker(\f_5)$. The exact sequence (2.2.4) \cite{drezet-maican}
yields the resolution
\[
0 \lra 6\O(-2) \stackrel{\psi'}{\lra} \O(-4) \oplus 15\O(-1) \stackrel{\f'}{\lra} 10\O \lra \G' \lra 0.
\]
We have $\h^0(\G')=10$, hence $\H^0(\G')=\H^0(\G)$.
The global sections of $\G$ generate $\O_C(3)$ and $\G'$ is generated by its global sections.
Thus $\G'=\O_C(3)$.
The maximal minors of any matrix representing $\f'$ generate the ideal of $C$
because the Fitting support of $\G'$ is $C$.
It is clear that $x=y$.
In conclusion, we have obtained the pair $(x, C) \in \Sigma$ from $\EE^1(\G)$
by performing algebraic operations.
\end{proof}

\begin{prop}
\label{6.3}
$X_6$ lies in the closure of $X_5$.
\end{prop}

\begin{proof}
Any generic sheaf $\O_C(2)(-P)$ in $X_6$, with $C \subset \P^2$ a smooth sextic curve
and $P \in C$, is the limit of a sequence of sheaves of the form $\O_C(2)(P_1-P_2 - P)$
as at \ref{5.3} (make $P_1$ converge to $P_2$).
\end{proof}


\section{The moduli space is the union of the strata}

\noi
In this final section we shall prove that $\M(6,2)$ is the union of the locally closed subsets
$X_1, \ldots, X_6$ we found above.

\begin{prop}
\label{7.1}
There are no sheaves $\F$ giving points in $\M(6,2)$ and satisfying the conditions
$\h^0(\F(-1))=0$, $\h^1(\F)=2$.
\end{prop}

\begin{proof}
Assume that there is such a sheaf $\F$.
Put $m=\h^0(\F \tensor \Om^1(1))$.
The Beilinson monad for the dual sheaf $\G=\F^\D(1)$ gives the resolution
\[
0 \lra 2\O(-2) \lra 4\O(-2) \oplus (m+2)\O(-1) \stackrel{\f}{\lra} m\O(-1) \oplus 4\O \lra \G \lra 0.
\]
Since $\G$ maps surjectively onto $\Coker(\f_{11})$, we have $m \le 3$.
The rest of the proof is exactly as at 3.1.3 \cite{mult_five}.
Let $\psi \colon 2\O(-2) \to (m+2)\O(-1)$ denote the morphism occurring in the above complex.
In the case $m=3$, say, there are three possible canonical forms for $\psi$ given at loc.cit.,
each leading to a contradiction.
\end{proof}

\begin{prop}
\label{7.2}
There are no sheaves $\F$ giving points in $\M(6,2)$ and satisfying the cohomological conditions
\[
\h^0(\F(-1)) \le 1, \quad \qquad \h^1(\F) \ge 3, \quad \qquad \h^1(\F(1)) = 0.
\]
\end{prop}

\begin{proof}
The argument is the same as at 7.2 \cite{mult_six_one} with notational differences only.
Assume that $\F$ gives a point in $\M(6,2)$ and satisfies the conditions
$\h^0(\F(-1))=0$, $\h^1(\F) \ge 3$. Write $p=\h^1(\F)$, $m=\h^0(\F \tensor \Om^1(1))$.
The Beilinson free monad for $\F$ reads
\[
0 \lra 4\O(-2) \oplus m\O(-1) \lra (m+2)\O(-1) \oplus (p+2)\O \stackrel{\psi}{\lra} p\O \lra 0,
\]
\[
\psi = \left[
\ba{cc}
\eta & 0
\ea
\right],
\]
and yields a resolution
\[
0 \lra 4\O(-2) \oplus m\O(-1) \stackrel{\f}{\lra} \Ker(\eta) \oplus (p+2)\O \lra \F \lra 0
\]
in which $\f_{12}=0$. We have $m+2-p = \rank(\Ker(\eta)) \le 4$ because $\f$ is injective.
Thus
\[
\h^0(\F(1)) = 3(p+2) + \h^0(\Ker(\eta)(1)) -m \ge 2(p+2) \ge 10
\]
forcing $\h^1(\F(1)) \ge 2$.
Assume, instead, that $\h^0(\F(-1))=1$.
The Beilinson monad for the dual sheaf $\G = \F^\D(1)$ reads
\[
0 \lra p\O(-2) \lra (p+2)\O(-2) \oplus (m+2)\O(-1) \lra m\O(-1) \oplus 5\O \lra \O \lra 0
\]
and leads to a resolution
\begin{multline*}
0 \lra p\O(-2) \lra \O(-3) \oplus (p-1)\O(-2) \oplus (m+2)\O(-1) \stackrel{\f}{\lra} \\
(m-3)\O(-1) \oplus 5\O \lra \G \lra 0
\end{multline*}
in which $\f_{13} =0$. Since $\G$ maps surjectively onto $\Coker(\f_{11},\f_{12})$, we have $m-3 \le p$.
Dualising the above resolution we get a monad for $\F$ of the form
\[
0 \lra 5\O(-2) \oplus (m-3)\O(-1) \lra (m+2)\O(-1) \oplus (p-1)\O \oplus \O(1) \stackrel{\psi}{\lra}
p\O \lra 0,
\]
\[
\psi = \left[
\ba{ccc}
\eta & 0 & 0
\ea
\right].
\]
The exact sequence
\[
0 \lra 5\O(-2) \oplus (m-3)\O(-1) \lra \Ker(\eta) \oplus (p-1)\O \oplus \O(1) \lra \F \lra 0
\]
gives the estimate
\[
\h^0(\F(1)) = 3(p-1) + 6 + \h^0(\Ker(\eta)(1)) - (m-3) \ge 3p+6-m \ge 2p+3 \ge 9.
\]
We deduce that $\h^1(\F(1)) \ge 1$.
\end{proof}

\begin{prop}
\label{7.3}
Let $\F$ be a sheaf giving a point in $\M(6,2)$ and satisfying the condition $\h^1(\F(1))=0$.
Then $\h^0(\F(-1))=0$ or $1$.
\end{prop}

\begin{proof}
Let $\F$ give a point in $\M(6,2)$ and satisfy the condition $\h^0(\F(-1))\ge 2$.
As at 2.1.3 \cite{drezet-maican}, there is an injective morphism $\O_C \to \F(-1)$
for a curve $C \subset \P^2$. This curve has degree $5$ or $6$, otherwise $\O_C$
would destabilise $\F(-1)$.
Assume that $\deg(C)=5$. The quotient sheaf $\CC = \F/\O_C(1)$ has Hilbert polynomial
$\PP(t)=t+2$ and zero-dimensional torsion $\TT$ of length at most $1$
(the pull-back in $\F$ of $\TT$ would be a destabilising subsheaf if its length were at least $2$).
If $\TT=0$, then $\CC \isom \O_L(1)$ for a line $L \subset \P^2$.
We get that $\h^0(\F(-1))=2$ and that the morphism $\O(1) \to \O_L(1)$ lifts to a morphism
$\O(1) \to \F$. The horseshoe lemma leads to the resolution
\[
0 \lra \O(-4) \oplus \O \lra 2\O(1) \lra \F \lra 0.
\]
Thus $\h^1(\F(1))=1$. Assume now that $\TT$ has length $1$.
Let $\F' \subset \F$ be the pull-back of $\TT$. According to 3.1.5 \cite{mult_five},
we have $\h^0(\F'(-1)) =1$. Since $\F/\F' \isom \CC/\TT \isom \O_L$
for a line $L \subset \P^2$, we get $\h^0(\F(-1))=1$, contradicting our choice of $\F$.

Assume now that $C$ is a sextic curve.
The quotient sheaf $\TT = \F/\O_C(1)$ is zero-dimensional of length $5$.
Let $\TT' \subset \TT$ be a subsheaf of length $4$ and let $\F'$ be its preimage in $\F$.
We claim that $\F'$ gives a point in $\M(6,1)$. If this were not the case, then $\F'$ would have
a destabilising subsheaf $\F''$, which may be assumed to be semi-stable.
By proposition \ref{2.3}, $\F$ is stable, so we have the inequalities $1/6 < \pp(\F'') < 1/3$.
This leaves only two possibilities: that $\F''$ give a point in $\M(5,1)$ or in $\M(4,1)$.
In the first case $\F/\F''$ is isomorphic to the structure sheaf of a line,
hence $\h^0(\F(-1)) = \h^0(\F''(-1)) = 0$ or $1$, cf. \cite{mult_five}.
This contradicts our choice of $\F$.
In the second case $\F/\F''$ is easily seen to be semi-stable,
hence it is isomorphic to the structure sheaf of a conic curve.
We get $\h^0(\F(-1)) = \h^0(\F''(-1))=0$, cf. \cite{drezet-maican},
contradicting our choice of $\F$. This proves the claim, i.e. that $\F'$ is semi-stable.
We have $\h^0(\F'(-1)) \ge 1$ so, according to \cite{mult_six_one},
there are two possible resolutions for $\F'$:
\[
0 \lra 2\O(-3) \oplus \O(-2) \lra \O(-2) \oplus \O(-1) \oplus \O(1) \lra \F' \lra 0
\]
or
\[
0 \lra \O(-4) \oplus \O(-1) \lra \O \oplus \O(1) \lra \F' \lra 0.
\]
Assume that $\F'$ has the first resolution. We apply the horseshoe lemma to the extension
\[
0 \lra \F' \lra \F \lra \C_x \lra 0,
\]
to the given resolution of $\F'$ and to the resolution
\[
0 \lra \O(-1) \lra 2\O \lra \O(1) \lra \C_x \lra 0.
\]
As $\h^0(\F'(-1)) =1$ and $\h^0(\F(-1)) \ge 2$, the morphism $\O(1) \to \C_x$ lifts to a morphism
$\O(1) \to \F$. We obtain the resolution
\[
0 \lra \O(-1) \lra 2\O(-3) \oplus \O(-2) \oplus 2\O \lra \O(-2) \oplus \O(-1) \oplus 2\O(1) \lra \F \lra 0.
\]
The morphism $\O(-1) \to 2\O(-3) \oplus \O(-2)$ occurring above is zero and
$\Ext^1(\C_x, \O(-2) \oplus \O(-1) \oplus \O(1))=0$, so we can argue as at 2.3.2 \cite{mult_five}
to conclude that $\F$ is a trivial extension of $\C_x$ by $\F'$.
This contradicts the semi-stability of $\F$.
Assume, finally, that $\F'$ has the second resolution.
We can apply the horseshoe lemma as above, leading to the resolution
\[
0 \lra \O(-1) \lra \O(-4) \oplus \O(-1) \oplus 2\O \lra \O \oplus 2\O(1) \lra \F \lra 0.
\]
We see from this that $\h^1(\F(1))=1$.
\end{proof}

\end{document}